\documentclass[11pt,reqno]{amsart}
\usepackage[all]{xy}
\usepackage[latin1]{inputenc} 
\usepackage[T1]{fontenc}
\usepackage[dvips]{graphics,graphicx}
\usepackage{amsfonts,mathabx}
\usepackage{amssymb}
\usepackage{amsmath}
\usepackage{mathrsfs, mathtools}
\usepackage{esint}
\usepackage{array, makecell} 
\allowdisplaybreaks

\usepackage{amsbsy,
	amsopn,
	amscd, 
	amsxtra, 
	amsthm,
	verbatim}
\usepackage{upref}
\usepackage[dvipsnames]{xcolor}
\usepackage{enumerate}
\usepackage[colorlinks,
linkcolor=red,
anchorcolor=red,
citecolor=red
]{hyperref}
\usepackage{cleveref}

\usepackage{thmtools}
\usepackage{caption}
\usepackage{subcaption}
\usepackage{verbatim}
\usepackage{booktabs}
\usepackage{longtable}

\newtheorem{theorem}{Theorem}

\newtheorem{lemma}{Lemma}[section]

\newtheorem*{proposition*}{Proposition}

\newtheorem*{corollary*}{Corollary}
\newtheorem{assumption}{Assumption}

\theoremstyle{definition}
\newtheorem{remark}[lemma]{Remark}
\newtheorem{example}[lemma]{Example}

\newcommand{\lihan}[1]{\textcolor{red}{#1}}
\newcommand{\pierre}[1]{\textcolor{blue}{#1}}

\newcommand{\po}{\left(}
\newcommand{\pf}{\right)}

\usepackage{geometry}
\newcommand{\R}{\mathbb{R}}
\newcommand{\dd}{\,\mathrm{d}}
\newcommand{\E}{\mathbb{E}}
\newcommand{\pp}{\mathbb{P}}
\newcommand{\ent}{\mathrm{Ent}}
\newcommand{\KL}{\mathrm{KL}}
\newcommand{\coot}{\mathfrak{C}_{\mathrm{OT}}}
\newcommand{\bangle}[1]{\langle #1\rangle}

\title[Entropic Convergence for PDMP]{On the entropic convergence for piecewise deterministic samplers: speedup and obstruction}
\author{Pierre Monmarch\'e}
\address{LAMA, Universit\'e Gustave Eiffel, 77420 Champs-sur-Marne, France.}
\email{pierre.monmarche@univ-eiffel.fr}
\author{Lihan Wang} \address{Department of Mathematics, National University of Singapore, 119076 Singapore}
\email{lihanw@nus.edu.sg}
\date{\today}

\begin{document}

\begin{abstract}
For piecewise deterministic samplers such as  Randomized Hamiltonian Monte Carlo (RHMC),  Bouncy Particle Sampler (BPS) or Zig-Zag Process (ZZP),    long-time exponential convergence rates have been established in previous works using Harris or $L^2$ hypocoercivity approaches. In particular, in the $L^2$ framework, a so-called \emph{diffusive-to-ballistic} speedup was known for log-concave targets, according to which the convergence rates of these samplers, with suitable parameters, are quadratically improved with respect to the standard overdamped Langevin diffusion process. A recent work by Jianfeng Lu showed that this speedup also holds for the kinetic Langevin diffusion process when the convergence is stated in terms of relative entropy, raising the question whether this also holds for piecewise deterministic samplers. The present work provides a positive and a negative answer to this: first, we show that the speedup holds in entropy for RHMC; second, we show that for BPS or ZZS, even for a standard Gaussian target, a similar result cannot hold, and even that exponential convergence (at any rate) in entropy fails. 
\end{abstract}

\maketitle

\section{Introduction and Main Results}

\subsection{Settings and objective}

Among Markov Chain Monte Carlo (MCMC) methods, an important family of samplers is given by kinetic processes, of the form $(X_t,V_t)$ with $X_t$  the position and $V_t = \dot X_t$ the velocity. The dynamics are designed such that, at equilibrium $X_t$ and $V_t$ are independent, respectively distributed according to the target measure $\mu_x\propto \exp(-U)$ and some auxiliary velocity equilibrium $\kappa$ (often a Gaussian distribution). Velocity playing the role of an additional auxiliary variable,  kinetic samplers are an important class among so-called lifted processes, which provide general constructions to design non-reversible samplers  \cite{diaconis2000analysis,turitsyn2011irreversible,vucelja2016lifting}. The motivation for such methods is to avoid the random walk (a.k.a. diffusive) behavior which is inherently associated to the reversibility condition  (at least for processes with small local steps). In particular,  by keeping the same direction for several steps due to their inertia, kinetic processes exhibit ballistic behavior (i.e. cover a distance of order $k$ in $k$ steps, instead of $\sqrt{k}$ as random walks), at least over the velocity correlation time-scale. This is known to improve the convergence rate to equilibrium with respect to standard reversible processes in some cases, with suitable choices of  parameters. See \cite{diaconis2000analysis, chen1999lifting} for processes in discrete spaces, \cite{cao2023explicit,  lu2022explicit,EberleLift} for processes in continuous spaces, as well as \cite{EberleLift,altschuler2025shifted,MonmarcheNesterov} for recent discussions on this topic and further references.

\smallskip

An important kinetic sampler is the (underdamped/kinetic) Langevin diffusion, which is in particular broadly used in molecular simulations~\cite{lelievre2016partial}. Its diffusive-to-ballistic speed-up is understood in a series of works, including \cite{cao2023explicit, EberleLift, fan2026sharp, lu2026sharp}. However, in the present work, we focus on non-diffusive processes, which are piecewise-deterministic Markov processes. These so-called velocity jump processes~\cite{MRZ} follow a deterministic dynamics (usually an Hamiltonian dynamics or even simple free transport) between random jumps of the velocity. A convenient way to describe the specific dynamics of such a Markov process is to provide its infinitesimal generator \cite{durmus2021piecewise}. In this work we consider the three most classical examples of such processes:
\begin{itemize}
    \item The \emph{Randomized Hamiltonian Monte Carlo} (RHMC) dynamics \cite{bou2017randomized}, with generator
    \begin{equation}
        \label{eq:generatorRHMC}
        \mathcal L_{RHMC} f = v\cdot \nabla_x f  - \nabla_x U(x) \cdot \nabla_v f + \gamma (\Pi_v f - f),
    \end{equation}
    where $\Pi_v \varphi(x) = \int_{\R^d}\varphi(x,w)\kappa(\dd w)$, with velocity equilibrium 
    \begin{equation}
        \label{eq:kappa=Gaussian}
        \dd \kappa  = (2\pi)^{-\frac{d}{2}}\exp\po -\frac{|v|^2}{2}\pf \dd v\,.
    \end{equation}
    The equation $\partial_t g= \mathcal{L}_{RHMC}^* g$ is also known as the \emph{linear Boltzmann equation}.
    \item The \emph{Zig-Zag process} (ZZP) \cite{bierkens2019zig}, with generator 
        \begin{equation}
        \label{eq:generatorZZP}
        \mathcal L_{ZZP} f= v\cdot \nabla_x f  + \sum_{i=1}^d (v_i \partial_{x_i} U)_+ (f(x,v-2v_i e_i) - f) + \gamma (\Pi_v f- f),
    \end{equation}
    where $e_i$ is the $i^{th}$ vector of the canonical basis,  $(s)_+=\max\{s,0\}$ and the velocity equilibrium $\kappa$ is either uniform over $\{-1,1\}^d$ (the state space of the process is then $\R^d\times\{-1,1\}^d$) or standard Gaussian \eqref{eq:kappa=Gaussian}.
    \item The \emph{Bouncy Particle Sampler} (BPS) \cite{bouchard2018bouncy} with generator
           \begin{equation}
        \label{eq:generatorBPS}
        \mathcal{L}_{BPS}f := v\cdot \nabla_x f +(v\cdot \nabla U)_+ (\mathcal{B} f-f) +\gamma(\Pi_v f-f),    \end{equation} 
   where 
\[ \mathcal{B} f(t,x,v) =f(t,x,v-2(v\cdot n)n), \quad n = \frac{\nabla U}{|\nabla U|}. \]
Here, the velocity equilibrium $\kappa$ can be any rotation invariant measure over $\R^d\setminus\{0\}$ (typically, the isotropic Gaussian distribution~\eqref{eq:kappa=Gaussian}, or a uniform measure over a ball or a sphere).
\end{itemize}
In each of the cases~\eqref{eq:generatorRHMC}, \eqref{eq:generatorZZP} and \eqref{eq:generatorBPS}, the term $\gamma (\Pi_v f-f)$ corresponds to velocity refreshments: at constant rate $\gamma>0$, the current velocity is discarded and a new velocity is sampled at equilibrium. This is meant to ensure the ergodicity of the process: without it, the Hamiltonian dynamics in~\eqref{eq:generatorRHMC} leaves invariant any density which is a function of the Hamiltonian $H(x,v)=U(x)+\frac12|v|^2$, and the BPS may also have irreducibility issues~\cite{bouchard2018bouncy}. In some situations, the ZZP might still be irreducible even when $\gamma=0$ \cite{bierkens2019ergodicity}, but in practice it is also always used with $\gamma>0$.

\smallskip

Both ZZP and BPS have been introduced as continuous-time limits of non-reversible lifted Markov chains as in~\cite{diaconis2000analysis}, see \cite{PetersdeWith,10.1214/16-AAP1217,M24}.  Contrary to RHMC for which the Hamiltonian dynamics is approximated in practice by a Verlet integration, ZZP and BPS can be simulated exactly in continuous-time resorting to Poisson thinning~\cite{bouchard2018bouncy}.

\medskip 

Before discussing the existing literature in more detail in the next section, let us first  simply say that, in the recent work~\cite{lu2026sharp}, Jianfeng Lu establishes the ballistic speed-up for the Langevin diffusion with the convergence stated in terms of relative entropy (a.k.a.\ Kullback-Leibler (KL) divergence). The question addressed in the present work is whether the same result hold for RHMC, ZZP and BPS.

\medskip 

The rest of this work is organized as follows. In the remaining of this introduction, after discussing previous works, we state our two main results, Theorem~\ref{thm:rhmc} (for RHMC) and Theorem~\ref{thm:noquantconv} (for ZZP and BPS). These two results are then respectively proven in Sections~\ref{sec:rhmc} and~\ref{sec:nonconvpdmp}. Section~\ref{sec:conclusion} provides some perspectives for future works.

 \subsection{Known results}\label{subsec:previouswork} To motivate our study and emphasize the specificity of our work, let us discuss existing convergence rates for kinetic samplers.

 \smallskip

Among kinetic processes, before being transferred to other processes, convergence results have always been obtained first on the kinetic Langevin diffusion. This has been done with various approaches, which can generally be sorted into two categories: the probabilistic approaches, which use stochastic analysis, including Lyapunov functions \cite{talay2002stochastic,mattingly_ergodicity_2002} or coupling approach \cite{eberle_couplings_2019} (providing results in weighted total variation norms or Wasserstein distances), or the analytic approaches, which use PDE tools to study the long-time convergence of the corresponding Kolmogorov/Fokker-Planck equation. Among analytic approaches, there are H{\'e}rau-Nier hypoellipticity arguments~\cite{herau2004isotropic}, Villani's hypocoercivity framework \cite{villani_hypocoercivity_2009}, the modified $L^2$ framework \cite{dolbeault_hypocoercivity_2015} (which we refer to as the DMS framework), and the space-time Poincar\'e inequality approach in \cite{albritton2024variational}. While many of these approaches fail to yield sharp convergence rates of the dynamics, recent progress has been made in this direction, starting from \cite{cao2023explicit, EberleLift} using the approach of \cite{albritton2024variational}, and \cite{fan2026sharp} which refines the DMS estimates. These works manage to capture the diffusive-to-ballistic speed-up as the convergence rates obtained in these works achieve a square-root acceleration compared to the corresponding reversible process when $U$ is convex. However, all these works derive convergence in $L^2(\mu)$ (where $\mu = \mu_x \otimes \kappa$), which is equivalent to $\chi^2(p_t\|\mu):=\int(\frac{\dd p_t}{\dd \mu}-1)^2 \dd \mu$ converging to zero, instead of relative entropy. Contrary to $\chi^2$, the relative entropy is suitable for the study of high-dimensional problems such as cut-off phenomenon~\cite{chafai2024cutoff} or mean-field particle systems  and their non-linear limits~\cite{MonmarcheMFkin,10.1214/24-EJP1079,MonmarcheNesterov}, which is related to its linear scaling  in terms of the dimension (while $\chi^2(\nu^{\otimes N}\|\mu^{\otimes N})$ grows exponentially with $N$ whenever $\nu\neq \mu$). Moreover,  results in relative entropy are amenable to numerical analysis for discretized schemes~\cite{vempala2019rapid,Chatterji,Camrudetal}. This is not the case of results in $\chi^2$, which is why, as can be seen in~\cite{altschuler2025shifted} the ballistic speed-up cannot be obtained for practical schemes directly from the $L^2$ results from~\cite{cao2023explicit, EberleLift}  proven for the continuous-time process (as it would be from the same results in entropy, cf. \cite[Remark 1]{MonmarcheNesterov} and the discussion after Theorem 5.11 in \cite{altschuler2025shifted}).  Besides, while entropic long-time convergence was known for the Langevin diffusion since \cite{villani_hypocoercivity_2009}, the rates where not sharp in the convex case, and the ballistic acceleration was only proved in the recent breakthrough \cite{lu2026sharp}. 

\smallskip 

 Several exponential convergence results have been obtained for PDMP samplers, including RHMC, ZZP and BPS, with coupling and Lyapunov techniques \cite{M24,durmus2020geometric,bierkens2019ergodicity} or $L^2$ approaches \cite{herau2006hypocoercivity,andrieu2021hypocoercivity,10.1214/20-AAP1659,lu2022explicit,MRZ,EberleLift}. Some of these results will be discussed further in Section~\ref{subsec:previous2} in  light of  our results. The results of~\cite{lu2022explicit,EberleLift} provide sharp rates. From these positive results in $L^2$, one could expect that sharp and explicit entropic convergence rates for PDMPs should be achievable, following the approach of \cite{lu2026sharp}. This is the main motivation of our work.

\subsection{Notations and results}

Throughout we do not distinguish between a probability measure and its density if it is absolutely continuous with respect to the Lebesgue measure. Let $p_t$ denote the law of $(X_t,V_t)$ along the dynamics, and let $g_t := \frac{\dd p_t}{\dd \mu_x \otimes \kappa}$ be the relative density, then $p_t$ satisfies the forward Kolmogorov equation
\begin{equation}\label{eq:flowpt}
    \partial_t p_t = \mathcal L^\dagger  p_t,
\end{equation}
where $\mathcal{L}^\dagger$ denotes the $L^2(\R^d \times \R^d)$-dual operator of the generator $\mathcal{L}\in\{\mathcal L_{RHMC},\mathcal L_{ZZP},\mathcal L_{BPS}\}$. Meanwhile, $g_t$ satisfies \[ \partial_t g_t = \mathcal{L}^* g_t,\] where $\mathcal{L}^*$ is the $L^2(\mu)$-adjoint of $\mathcal{L}.$ We adapt the notation of \cite{lu2026sharp}: we define $q_t:=\Pi_v g_t$ to be the relative $x$-marginal of the PDMP, and $h_x(v):=g(x,v)/q(x)$ be the conditional density on $v$, defined for any $x$ such that $q(x)>0$. For probability densities $p$, we use $\KL(p\|\mu):= \int \log \frac{\dd p}{\dd \mu}\dd p$ or $\KL(q\mu_x\|\mu_x)$ to denote the full relative entropy or relative entropy in $x$-marginals. For relative densities $q(x), \, g(x,v), \, h_x(v)$, we use the notations $\ent_x(q):=\int q\log q\dd \mu_x$, $\ent_\kappa(h_x):= \int h_x\log h_x \dd \kappa(v)$, $\ent(g):= \int g\log g\dd \mu$ and $\ent_v(g) := \int q(x) \ent_\kappa(h_x) \dd \mu_x$. Note that $\KL(p_t\|\mu)=\ent(g) = \ent_x(q)+\ent_v(g)$.  For $p \geqslant 1$ we write $\mathcal W_p(\nu,\nu')$ the $L^p$-Wasserstein distance between two probability measures $\nu,\nu'$ over $\R^d$, given by
\[\mathcal W_p^p(\nu,\nu') = \inf_{\pi \in \Pi(\nu,\nu')} \int_{\R^{2d}} |x-y|^p \pi(\dd x\dd y)\,, \]
with $\Pi(\nu,\nu')$ the set of couplings of $\nu $ and $\nu'$.

\begin{assumption}\label{ass:JLass1}
The probability measure $\mu_x \propto \exp(-U(x))$ satisfies the logarithmic Sobolev inequality (LSI) with constant $\rho>0$:
\begin{equation}\label{eq:LSI}
    \ent_x(f) \le \frac{1}{2\rho}\int \frac{|\nabla f|^2}{f}\dd \mu_x, \quad \forall f\ge 0, \, \int f\dd \mu_x=1.
\end{equation}
Moreover, the Hessian of the potential $U \in C^\infty(\R^d)$ is bounded from below $ \nabla^2 U \ge -K\mathrm{Id}$ with $0\le K <\rho$, and all derivatives of second-order or above of $U$ are uniformly bounded in $\R^d$: $\sup_{x\in \R^d} |\nabla^\alpha U(x)|<\infty$ for any $|\alpha|\ge 2$.
\end{assumption}

\begin{remark}
    We would like to comment here that $0\le K <\rho$ is a natural condition for us to directly apply the approach of \cite{lu2026sharp} to obtain convergence in relative entropy. Indeed, it was proved in \cite{otto2000generalization} the following HWI inequality:
    \begin{equation}\label{eq:hwi} \ent_x(q) \le \mathcal{W}_2(q\mu_x,\mu_x)\sqrt{I(q\mu_x\|\mu_x)}+\frac{K}{2}\mathcal{W}_2^2(q\mu_x,\mu_x),\end{equation}
    where $I$ denotes the relative Fisher information. After a careful inspection on the proof of \cite{lu2026sharp}, we find that it only explicitly uses the Talagrand inequality, which is weaker than the LSI \cite{otto2000generalization}
\begin{equation}\label{eq:talagrand}  \mathcal{W}_2^2(q\mu_x,\mu_x) \le \frac{2}{\rho}\ent_x(q).\end{equation} The combination of \eqref{eq:talagrand} and \eqref{eq:hwi} then indicates
\[ \ent_x(q) \le \sqrt{\frac{2}{\rho}\ent_x(q)}\sqrt{I(q\mu_x\|\mu_x)}+\frac{K}{\rho}\ent_x(q),\] and hence Talagrand inequality implies LSI when $K<\rho$, which is consistent with our results since we expect LSI to be necessary for entropic exponential convergence to hold. 
\end{remark}
\begin{remark}
    The condition that all derivatives of second-order or above of $U$ are uniformly bounded is only stated for technical reasons and we expect that it can be significantly weakened (these bounds do not appear quantitatively in the result). We only used this condition for the regularity argument in the proof of Theorem \ref{thm:rhmc}.
\end{remark}

Our first main theorem states that RHMC also has entropic exponential convergence and obtains square-root acceleration with a careful choice of $\gamma$ and if the non-convexity of $U$ is mild. 
\begin{theorem}\label{thm:rhmc}
    Let $p_t$ satisfies \eqref{eq:flowpt} with initial condition $p_0$ and $\mathcal{L}$ given by RHMC. Assuming $U(x)$ satisfies Assumption \ref{ass:JLass1}. Let $\gamma = \Gamma \sqrt{\rho}$ with $\Gamma>0$ and choose $0<\theta \le \min \{ \frac{\Gamma}{18}, \frac{1}{4\Gamma}\po 1-\frac{K}{\rho}\pf \}$. Set $\lambda_\Gamma := \po 1-\frac{K}{\rho}\pf\frac{\theta}{2(1+\theta)}.$ Then, we have for every finite-entropy initial condition $p_0$,
    \[ \KL(p_t\|\mu) \le \frac{1+\theta}{1-\theta}\exp(-\lambda_\Gamma\sqrt{\rho}t)\KL(p_0\|\mu).\]
    In particular, when $K=0$, we recover the square-root entropic acceleration result as in \cite{lu2026sharp} for the Langevin diffusion.
\end{theorem}

Meanwhile, for ZZP and BPS, while existing results \cite{andrieu2021hypocoercivity, MRZ, lu2022explicit} proved their exponential convergence under other metrics, we show that such convergence cannot hold for relative entropy, even though the relative entropy decreases along both dynamics:
\begin{equation}\label{eq:ZZPKLdec}  \frac{\dd}{\dd t} \ent(g_t)=-\gamma  \int (g_t-\Pi_v g_t) \log g_t \dd \mu(x,v)-\sum_{k=1}^d\int (v_k\partial_k U)_+ g_t\Big( \frac{\mathcal{B}_kg_t}{g_t}-1-\log\frac{\mathcal{B}_kg_t}{g_t}\Big) \dd \mu \le 0, \end{equation}
\begin{equation}\label{eq:BPSKLdec} \frac{\dd}{\dd t} \ent(g_t)=-\gamma \int (g_t-\Pi_v g_t) \log g_t \dd \mu(x,v) -\int (v\cdot \nabla U)_+ g_t\Big( \frac{\mathcal{B}g_t}{g_t}-1-\log\frac{\mathcal{B}g_t}{g_t}\Big) \dd \mu \le 0, \end{equation}
Note that in Section \ref{sec:rhmc}, we prove that
\[ \int (g_t-\Pi_v g_t) \log g_t \dd \mu(x,v) \ge \ent_v(g_t):=\int q_t(x)\ent_{\kappa}(h_{t,x}) \dd \mu_x\ge 0. \]

In the next result, apart from relative entropy and Wasserstein distances, we also consider weighted total variation distances (a.k.a. $\mathcal V$-norms) defined as 
\begin{equation}
    \label{eq:def-Vnorm}
\|\nu - \nu'\|_{\mathcal V} := \sup\left\{ \int_{\R^{2d}}f \dd (\nu-\nu'),\, \|f/\mathcal V\|_\infty \leqslant 1\right\}\,,
\end{equation}
given some weight (or Lyapunov function) $\mathcal V :\R^d \rightarrow [1,\infty)$.

\begin{theorem}\label{thm:noquantconv}
    Let $p_t$ satisfies \eqref{eq:flowpt} with $\mathcal{L}$ given by ZZP or BPS. Suppose $\kappa$ is either supported on a compact subset of $\R^d$, or the standard Gaussian distribution on $\R^d$. Then, for $\mu_x$ being the standard Gaussian distribution, there does not exist an $\varepsilon>0$ and a function $\lambda(t)>0$ decreasing to zero, such that for all initial conditions $p_0$ with $\KL(p_0\|\mu)\le \varepsilon$,
    \begin{equation}\label{eq:KLdecay}
        \KL(p_t\|\mu)\le \lambda(t)\KL(p_0\|\mu).
    \end{equation}
    More precisely, we have for all $t\ge 0$ and $\varepsilon>0$
    \begin{equation}
        \label{eq:lowerboundsTheorem2KL}
        \sup_{\KL(p_0\|\mu) \leqslant \varepsilon}\frac{\KL(p_t\|\mu)}{\KL(p_0\|\mu)} = 1\,.
    \end{equation}
    In addition, for $\mathcal{W}_p$ distance with $p\in [1,\infty)$ and for $\mathcal V$-norms with $\mathcal V $ having at most a polynomial growth, we have for all $t\ge 0$ and $\varepsilon>0$
    \begin{equation}
        \label{eq:lowerboundsTheorem2}
    \sup_{\mathcal{W}_p(p_0,\mu) \leqslant \varepsilon}\frac{\mathcal{W}_p(p_t,\mu)}{\mathcal{W}_p(p_0,\mu)} \ge 1\,,\qquad \sup_{\|p_0-\mu\|_{\mathcal V} \leqslant \varepsilon}\frac{\|p_t-\mu\|_{\mathcal V}}{\|p_0-\mu\|_{\mathcal V}} \ge  1\,.
    \end{equation}
    For a weight $\mathcal V$ satisfying $\mathcal V(x,v) \leqslant C e^{C|x|^\delta} + C e^{C|v|^2}$   for some $C>0$ and $\delta<1$, 
    the second inequality is also true if $\kappa$ is compactly supported, or if $\kappa$ is the standard Gaussian and $\delta<2/3$.
\end{theorem}
\begin{remark}
    From our negative result in Theorem \ref{thm:noquantconv}, we see that it is impossible to develop a unified entropic convergence framework for second-lifts of overdamped Langevin dynamics, unlike in the $L^2$ setting \cite{EberleLift, brigati2025hypocoercivity}.  
\end{remark}
 In Section \ref{sec:nonconvpdmp}, we will present some counterexamples that show why quantitative exponential convergence results cannot hold. We also compare our negative result with previous ones which, under the settings of Theorem~\ref{thm:noquantconv}, show an exponential decay either in $L^2$ norm (i.e. $\chi^2$ divergence) or $\mathcal V$-norms with Lyapunov functions $\mathcal V$ having an exponential growth. The fact that these metrics have significantly different behavior than the $\mathrm{KL}$ or $\mathcal W_p$ distances (in particular a different sensitivity to the tails) is related to a known topic in the analysis of MCMC methods: some result require warm start, i.e. assume that the initial distribution has a finite $\chi^2$ divergence with respect to the equilibrium (or a similar condition, for instance for some other R\'enyi divergence), and may fail otherwise. Theorem~\ref{thm:noquantconv} shows that this is the case for exponential convergence for the BPS and ZZP. We refer to \cite{2026arXiv260109019B,ZhaAltChe26HMC} for further details and references on this topic.

\section{Proof of Theorem \ref{thm:rhmc}: Randomized HMC}\label{sec:rhmc}
Let us explicitly write out the PDE satisfied by $g_t$:
\begin{equation}\label{eq:RHMC} \partial_t g_t + v\cdot \nabla_x g_t - \nabla U \cdot \nabla_v g_t = \gamma (\Pi_v g_t -g_t).\end{equation}
Our proof of Theorem \ref{thm:rhmc} uses the approach of \cite{lu2026sharp}. We will highlight the differences of the proof. We remind the readers that
\[\KL(p_t\|\mu) = \ent(g_t) = \ent_x(q_t) + \ent_v(g_t).\]

We first obtain the proof for regular finite-entropy solutions of \eqref{eq:RHMC}, in the sense of \cite[Definition 2.5]{lu2026sharp}. Let us start with entropy dissipation
\begin{equation}\label{eq:entdiss}
    \frac{\dd}{\dd t} \ent(g_t)= \gamma\int (\Pi_v g_t-g_t) \log g_t \dd \mu(x,v) =: -\gamma \mathcal{D}(g_t).
\end{equation}
We claim that
\[ \mathcal{D}(g) \ge \ent_v(g). \]
This is equivalent to
\[ \int (g-\Pi_v g) \log g \dd \mu \ge \int g\log g\dd \mu-\int \Pi_v g \log \Pi_v g \dd \mu_x, \] or \[ \int \Pi_v g \log \Pi_v g \dd \mu_x \ge \int \Pi_v g \log g \dd \mu = \int \Pi_v g(\log \Pi_v g + \log h_x(v)) \dd \mu,\] which is true since by Jensen's inequality,\[ \int \Pi_v g \log h_x(v) \dd \mu = \int \Pi_v g \Big(\int \log h_x(v) \dd \kappa(v) \Big)\dd \mu_x \le \int \Pi_v g \log\Big(\int  h_x(v) \dd \kappa(v) \Big)\dd \mu_x=0. \]

An important element of the approach in \cite{lu2026sharp} is the modified energy \[ \mathcal{H}_\epsilon(g):= \ent(g_t)+\epsilon \coot(g_t),\]where $\coot(g):= \int j(x) \cdot \xi_q(x) \dd \mu_x$, $j:=\Pi_v (vg)$ is the flux, and $\xi_q(x)=x-T_q(x)$, where $T_q(x)$ is the optimal transport map from $q\mu_x$ to $\mu_x$. Testing \eqref{eq:RHMC} against 1 and $v$ with respect to $L^2(\kappa)$, and we get the continuity equations in the $x$ variable:
\begin{equation}\label{eq:conteqx}
    \partial_t q =\nabla_x^* j, \quad \partial_t j = -\nabla_x q + \nabla_x^* (M-q\mathrm{Id})-\gamma j,
\end{equation}  where $M=\Pi_v(v\otimes v g)$. Since these equations are identical to the case of \cite{lu2026sharp}, the entire argument of \cite{lu2026sharp} passes through when $U$ is convex, with $\mathcal{D}(g)$ playing the role of $\frac{1}{2}I_v(g)$. 

\smallskip

We nevertheless present the rest of the proof when $U$ is not convex but allows a small nonconvexity in the sense of $\nabla^2 U \ge -K \mathrm{Id}$, and treat the convex case as a special case when $K=0$. Then, we have instead
\[ \nabla U(x) \cdot (x-T (x))-U(x)+U(T(x)) \ge  -\frac{K}{2}|x-T(x)|^2,\] the integration of which against $q\dd \mu_x$ is exactly $-\frac{K}{2}\mathcal{W}_2^2(q\mu_x,\mu_x)$. Substituting this, \cite[Lemma 5.3]{lu2026sharp} and \cite[Lemma 3.1]{lu2026sharp} into \cite[(4.10)]{lu2026sharp} then yields
\begin{align*}\frac{\dd}{\dd t}\coot(g_t) & \le  -\gamma \coot(g_t) +6\ent_v(g_t)-\ent_x(q_t) + \frac{K}{2}\mathcal{W}_2^2(q\mu_x,\mu_x) \\ & \le -\gamma \coot(g_t) +6\ent_v(g_t)-\ent_x(q_t) + \frac{K}{\rho}\ent_x(g_t). \end{align*}
Here we used Talagrand inequality \eqref{eq:talagrand} in the latter step. Combined with \eqref{eq:entdiss}, we arrive at
\begin{align*}
    \frac{\dd}{\dd t}\mathcal{H}_\epsilon(g_t) & \le -\gamma \mathcal{D}(g_t) -\gamma\epsilon \coot(g_t) + 6\epsilon \ent_v(g_t) -\po 1-\frac{K}{\rho}\pf\epsilon \ent_x(g_t) 
    \\ & \le \gamma\epsilon \sqrt{\frac{4}{\rho}\ent_v(g_t)\ent_x(q_t)} - (\gamma-6\epsilon )\ent_v(g_t) -\po 1-\frac{K}{\rho}\pf\epsilon \ent_x(g_t),
\end{align*}
where in the last step we used \cite[Lemma 3.2]{lu2026sharp} to bound the $\coot(g_t)$ term. Writing $\gamma = \Gamma \sqrt{\rho}$ and $\epsilon = \theta \sqrt{\rho}$, and we have
\begin{align*}
     \frac{\dd}{\dd t}\mathcal{H}_\epsilon(g_t) & \le \sqrt{\rho}\Big(2\Gamma \theta \sqrt{\ent_v(g_t)\ent_x(q_t)}- (\Gamma-6\theta)\ent_v(g_t) -\po 1-\frac{K}{\rho}\pf\theta \ent_x(g_t)\Big) \\ & \le -\sqrt{\rho}\Big(\sqrt{1-\frac{K}{\rho}}\sqrt{\Gamma\theta\ent_v(g_t)\ent_x(q_t)}-\frac{2\Gamma}{3}\ent_v(g_t) - \po 1-\frac{K}{\rho}\pf\theta \ent_x(g_t)\Big) \\ & \le -\sqrt{\rho}\Big(-\frac{\Gamma}{6}\ent_v(g_t) - \po 1-\frac{K}{\rho}\pf\frac{\theta}{2} \ent_x(g_t)\Big) \\ & \le -\sqrt{\rho}\po 1-\frac{K}{\rho}\pf\frac{\theta}{2}\ent(g_t) \le -\sqrt{\rho}\po 1-\frac{K}{\rho}\pf\frac{\theta}{2(1+\theta)}\mathcal{H}_\epsilon(g_t).
\end{align*}
Here in the second line we used that $\theta \le \min \{ \frac{\Gamma}{18}, \frac{1}{4\Gamma}( 1-K/\rho) \}$. We finish the proof for regular finite-entropy solutions using a Gr\"onwall argument and the equivalence of $\mathcal{H}_\epsilon$ and $\KL(p_t\|\mu)$ by \cite[Lemma 3.2]{lu2026sharp}. 

\medskip

It remains to remove the regularity requirements of the solution, for which we again roughly follow the proof in \cite[Section 7]{lu2026sharp}, with the major difference that short-time hypoelliptic regularization in the sense of \cite{herau2004isotropic} no longer holds for RHMC due to the velocity diffusion operator being replaced by $\gamma(\Pi_v-\mathrm{Id})$. Fix any smooth and compactly supported $C^\infty$ initial condition $g_0$ with $g_0\ge 0, \int g_0\dd\mu=1$, and fix $\zeta\in (0,1)$. Consider the solution of \eqref{eq:RHMC} with initial condition $\tilde g_0=\zeta + (1-\zeta)g_0$, then we have by linearity $\tilde{g}_t = \zeta + (1-\zeta)g_t$. Since $\mu$ is invariant, we get $\zeta  \leqslant \tilde g_t \leqslant \zeta + (1-\zeta)\|g_0\|_\infty$ for all $t\geqslant 0$. We denote $\tilde{p}_t:= \tilde{g}_t\mu$. To show that $\tilde g_t$ is a regular solution, we consider $f_t := \exp(\frac{U(x)}{2}+\frac{|v|^2}{4})g_t$, then $f_0$ is smooth and compactly supported, and $f_t$ satisfies \begin{equation}\label{eq:halfweight}  \partial_t f_t = -v\cdot\nabla_x f_t +\nabla U \cdot \nabla_v f_t + \gamma\Big(\sqrt{\kappa(v)}\int f_t \dd \sqrt{\kappa(v)}-f_t\Big).\end{equation} Define \[M_{N,m}(f_t) = \max_{|\alpha|+|\beta|\le m}\sup_{x,v\in \R^d}(1+|x|+|v|)^N |\partial_x^\alpha\partial_v^\beta f_t(x,v)|.\] Thanks to Assumption \ref{ass:JLass1}, $\nabla U$ is Lipschitz in $\R^d$ and has all bounded derivatives, therefore by standard ODE theory, the particle flow $\Phi_t$ generated by Hamiltonian transport $-v\cdot \nabla_x + \nabla U \cdot \nabla_v$ is well-posed and smooth with all derivatives uniformly bounded in any finite time interval, and hence $M_{N,m} \big(f_0(\Phi_{-t}(x,v))\big)$ cannot grow more than exponentially fast in $t$. Meanwhile, the projection operator $\tilde{\Pi}_v f :=\sqrt{\kappa(v)}\int f \dd \sqrt{\kappa(v)} $ satisfies 
\[\partial_x^\alpha \partial_v^\beta \tilde{\Pi}_v f = \partial_v^\beta \sqrt{\kappa(v)} \int \partial_x^\alpha f \dd \sqrt{\kappa(v)},\] and since $\partial_v^\beta \sqrt{\kappa(v)} $ is a polynomial multiplied by Gaussian, we observe that $\tilde{\Pi}_v$ is also bounded in the Schwartz space. This shows that the generator of \eqref{eq:halfweight} leaves the Schwartz space invariant, and by a Gr\"onwall argument, for any $N,m$ we can find some $\Lambda_{N,m}$ such that  
\[M_{N,m}(f_t) \le \exp(\Lambda_{N,m} t) M_{N,m}(f_0),\] which gives the same desired tail bound similar to \cite[Lemma 7.1]{lu2026sharp}. Hence, following the proof of \cite[Proposition 7.2]{lu2026sharp} (in particular, the continuity equations for the $x$-marginals \eqref{eq:conteqx} are identical), we observe that $\{\tilde g_{t}\}$ is regular on each interval $[0,T]$. This allows us to bound
\[\KL(\tilde p_t\|\mu) \leqslant  \frac{1+\theta}{1-\theta}\exp(-\lambda_\Gamma\sqrt{\rho}t)\KL(\tilde p_0\|\mu) \leqslant  \frac{1+\theta}{1-\theta}\exp(-\lambda_\Gamma\sqrt{\rho}t)\KL(p_0\|\mu)\,.\]
The rhs is independent of $\zeta$, and by lower semi-continuity of $\KL$ by letting $\zeta\rightarrow 0$ we can replace the lhs of above by $\KL(p_t\|\mu)$. We hence proved Theorem \ref{thm:rhmc} for any smooth and compactly supported initial data $g_0$. To prove the theorem for general initial condition with finite entropy, we resort to the approximation argument in \cite[Lemma 7.3]{lu2026sharp}. 


\section{Proof of Theorem \ref{thm:noquantconv}: No Quantitative Entropic Convergence} \label{sec:nonconvpdmp}

Throughout, we assume the target distribution is $\dd \mu_x = (2\pi)^{-\frac{d}{2}}\exp(-\frac{|x|^2}{2}) \dd x$, although the proof can easily be generalized to arbitrary $\mu_x$. 
We remind the readers that both ZZP and BPS are stochastic processes $(X_t,V_t)$ with $\frac{\dd}{\dd t}X_t=V_t$. For all our constructions, we assume $p_0 = q_0\mu_x \otimes \kappa$ so that, in particular, $\KL(p_0\|\mu) = \KL(q_0\mu_x\|\mu_x)$.

\subsection{Preliminary lemmas}

Let us first notice that a bound of the form~\eqref{eq:KLdecay} with any $\lambda(t)$ going to zero as $t\rightarrow \infty$ in fact implies that the bound actually hold with $\lambda(t)$ decaying exponentially fast:

\begin{lemma}\label{lem:extrap}
    Let $p_t$ be the solution of some flow such that $\KL(p_t\|\mu)$ is nonincreasing. Suppose there exists some $\delta \in (0,1)$ and some $t_*>0$ such that for all initial conditions $p_0$ with $\KL(p_0\|\mu)<\infty$, we have decay estimate    \begin{equation}\label{eq:KLdecayinit}
      \KL(p_{t_*}\|\mu) \le \delta\KL(p_0\|\mu),  
    \end{equation}
    then, we have for all $t\ge 0$,
    \begin{equation}\label{eq:KLdecaygen}
        \KL(p_{t}\|\mu) \le \delta^{-1}\exp\Big(-\frac{t}{t_*}\log \frac{1}{\delta} \Big) \KL(p_0\|\mu),
    \end{equation}
\end{lemma}
\begin{proof}
    For arbitrary $t>0$, we can find some non-negative integer $n$ and some $s\in [0,t_*)$ such that $t=nt_*+s$, then we can apply \eqref{eq:KLdecayinit} repeatedly and obtain
    \[ \KL(p_t\|\mu) \le \KL(p_{nt_*}\|\mu) \le \delta^{n}\KL(p_0\|\mu) \le \delta^{\frac{t}{t_*}-1}\KL(p_0\|\mu),\] which is \eqref{eq:KLdecaygen}. 
\end{proof}

 Concerning the inequalities~\eqref{eq:lowerboundsTheorem2}, in general $\mathcal W_p(p_t,\mu)$ and $\|p_t - \mu\|_{\mathcal V}$ are not decreasing in $t$, but for any function $\mathcal F:\mathcal P(\R^d)\setminus\{\mu\} \rightarrow (0,\infty]$,
 \begin{equation}
     \label{eq:Fdecayexpo}
 \exists t_*,\varepsilon>0,\, \lambda_*:= \sup_{\mathcal F(p_0)\leqslant \varepsilon} \frac{\mathcal F(p_t)}{\mathcal F(p_0)}<1 \quad \Rightarrow\quad \mathcal F(p_{nt_*}) \leqslant \lambda^n \mathcal F(p_0)\,\forall n\in  \mathbb N,\ p_0 \in \mathcal P(\R^d)\ \text{with}\ \mathcal F(p_0) \leqslant \varepsilon\,.
 \end{equation}
 
Thanks to this first comment, we see that the inequalities~\eqref{eq:lowerboundsTheorem2KL} and \eqref{eq:lowerboundsTheorem2} will be proven if we can find initial distributions (arbitrarily close to $\mu$ in terms of $\mathrm{KL}$, $\mathcal W_p$ or $\|\cdot\|_{\mathcal V}$) for which the long-time decay of the corresponding metric is sub-exponential. To lower-bound these quantities  in terms of the tail  of $p_0$, using that, for any $\mathcal V\geqslant 1$, $p\geqslant 1$ and $\nu \in\mathcal P(\R^d)$,
\[\KL(\nu\|\mu) \geqslant \mathrm{TV}^2(\nu\| \mu) \,,\qquad \|\nu-\mu\|_{\mathcal V} \geqslant \mathrm{TV}(\nu\| \mu)\,,\qquad \mathcal W_p(\nu,\mu) \geqslant \mathcal W_1(\nu,\mu)\,,   \]
we will rely on the following lemma.  
\begin{lemma}\label{lem:tail}
    Let $p_t$ be the solution of \eqref{eq:flowpt} with $\mathcal{L}$ given by ZZP or BPS. Then for all $T, \, v_*>0$, we have estimates 
    \begin{equation}\label{eq:KLgtrtail}
        \mathrm{TV}(p_T\| \mu) \ge  \pp_{q_0\mu_x}(|X_0|>2Tv_*)-\pp\big(\sup_{t\in[0,T]} |V_t|>v_*\big)-\pp_{\mu_x}(|x|>Tv_*),
    \end{equation}
    \begin{equation}\label{eq:W1gtrtail}
        \mathcal W_1(p_T,\mu) \ge \pp_{q_0\mu_x}(|X_0|>2Tv_*+1) - \pp_{\mu_x}(|x|>Tv_*)-\pp\big(\sup_{t\in[0,T]} |V_t|>v_*\big) \,.
    \end{equation}
\end{lemma}
\begin{proof}
    The important estimate for \eqref{eq:KLgtrtail} is that total variation distance does not increase after marginalization: 
\begin{align*} \mathrm{TV}(p_T\| \mu) \ge \mathrm{TV}(q_T\mu_x\|\mu_x) & \ge \pp(|X_T|>Tv_*)-\pp_{\mu_x}(|x|>Tv_*).\end{align*}
From $\{|X_0|>2Tv_*\} \subseteq \{|X_T|>Tv_*\} \cup \{\sup_{t\in[0,T]}|V_t|> v_* \}$ (if the process started far, either it is still far at time $T$, or it got a high velocity at some point), we get
\begin{align*}
    \pp(|X_T|>Tv_*) 
    & \ge \pp_{q_0\mu_x}(|X_0|>2Tv_*)-\pp\po\sup_{t\in[0,T]} |V_t|>v_*\pf,
\end{align*}
which finishes the proof of \eqref{eq:KLgtrtail}. 
The proof for \eqref{eq:W1gtrtail} is similar: consider a function $f:\R\rightarrow \R$ which is zero in $[-Tv_*,Tv_*]$, $1$ outside $[-Tv_*-1,Tv_*+1]$ and affine in between, then $f$ is 1-Lipschitz, which allows us to apply the dual formulation of $\mathcal{W}_1$ distance:
\begin{align*}\mathcal W_1(p_T,\mu) &\geqslant \mathbb E(f(|X_T|)) - \E_{\mu_x}f(|x|) \\ & \geqslant \mathbb P(|X_T|>Tv_*+1) - \pp_{\mu_x}(|x|>Tv_*) \\ & \ge \pp_{q_0\mu_x}(|X_0|>2Tv_*+1) - \pp_{\mu_x}(|x|>Tv_*) - \pp\po\sup_{t\in[0,T]} |V_t|>v_*\pf\,.\end{align*}
\end{proof}
In our counterexamples, $q_0\mu_x$ has substantially heavier tail than $\mu_x$, thus the rhs of both \eqref{eq:KLgtrtail} and \eqref{eq:W1gtrtail} are strictly positive for any sufficiently large $T$, and asymptotically dominated by the term depending on $q_0\mu_x$.

\subsection{Proof of Theorem \ref{thm:noquantconv}: Bounded Velocity}\label{subsec:bddv}

As discussed with~\eqref{eq:Fdecayexpo}, in order to prove Theorem~\ref{thm:noquantconv}, it is sufficient to exhibit suitable counter-examples for which long-time convergence is at most algebraic. Let us first look at the proof when the support of $\kappa$ is bounded, and we use $v_*$ to denote the supremum of $|v|$ in the support of $\kappa$, which allows us to apply Lemma \ref{lem:tail} with $\pp\big(\sup_{t\in[0,T]} |V_t|>v_*\big)=0$. 

\begin{example}\label{ex:fattail}
Let us take $q_0\mu_x \, \propto \,\bangle{x}^{-(d+\beta)}$ where $\bangle{x}:=\sqrt{1+|x|^2}$. We note that \[\KL(q_0\mu_x\|\mu_x)=\int \bangle{x}^{-(d+\beta)}\Big(\frac{|x|^2}{2}-(d+\beta)\log\bangle{x}\Big)\dd x ,\] hence the sufficient and necessary condition for $\KL(q_0\mu_x\|\mu_x)<\infty$ is $\beta>2$. We can estimate the tail of $\mu_x$ as \begin{equation}\label{eq:muxtail}\pp_{\mu_x} (|x|> Tv_*) \sim \int_{Tv_*}^\infty r^{d-1}\exp(-\frac{r^2}{2})\dd r \sim (Tv_*)^{d-2}\exp\po-\frac{(Tv_*)^2}{2}\pf.\end{equation}
Meanwhile, 
\begin{equation}\label{eq:p0tail}\pp_{q_0\mu_x}(|X_0|>2Tv_*)\sim \int_{2Tv_*}^\infty r^{d-1}r^{-(d+\beta)}\dd r \sim (2Tv_*)^{-\beta},\end{equation} which decays to zero much slower than $\pp_{\mu_x} (|x|> Tv_*) $ as $T\to\infty$. Therefore, we may deduce from Pinsker's inequality and \eqref{eq:KLgtrtail} that there exists some constant $C$ such that for all $T>0$,
\begin{equation}\label{eq: kllowerbd} \KL(p_T\|\mu) \ge  \mathrm{TV}^2(p_T,\mu)\ge C(1+T)^{-2\beta}.\end{equation}
At this point, we can already prove that there is no time $t\geqslant 0$ such that \eqref{eq:KLdecay} holds uniformly over all initial distributions $p_0$ with some $\lambda(t)<1$. Indeed, otherwise, we can apply Lemma \ref{lem:extrap} with  $t_*=t$, $\delta = \lambda(t_*)$. Hence $s\mapsto \KL(p_s\|\mu)$ must decrease exponentially fast, which is a contradiction to \eqref{eq: kllowerbd}. 

\smallskip

This example can be slightly modified to show that a quantitative convergence result in the form of \eqref{eq:KLdecay} cannot be held locally in the sense that we restrict it to initial conditions with $\KL(p_0\|\mu)\le \varepsilon$ for some small $\varepsilon>0$. Indeed, let us consider a family of initial conditions $p_0^\alpha := (1-\alpha)\mu + \alpha p_0$, with $\alpha\in (0,1]$. Then by Jensen's inequality, we have 
\[ \KL(p_0^\alpha\|\mu) \le (1-\alpha)\KL(\mu\|\mu) + \alpha \KL(p_0\|\mu) = \alpha \KL(p_0\|\mu),\]
 which can be made arbitrarily small as $\alpha \rightarrow 0_+$. On the other hand, notice that
 \[\pp_{q_0^\alpha \mu_x}(|X_0|>2Tv_*) \ge \alpha \pp_{q_0\mu_x}(|X_0|>2Tv_*) \sim \alpha(2Tv_*)^{-\beta}, \] which, using again \eqref{eq:KLgtrtail}, leads to 
 \[ \KL(p_T\|\mu)\ge C\alpha^2 (1+T)^{-2\beta}\] for $T$ sufficiently large. We obtain  $\sup_{\KL(p_0\|\mu) \leqslant \varepsilon}\frac{\KL(p_t\|\mu)}{\KL(p_0\|\mu)} = 1\,$ as $\KL(p_t\|\mu)$ is nonincreasing thanks to \eqref{eq:ZZPKLdec} and \eqref{eq:BPSKLdec}.

 \smallskip

The same argument for $\mathcal{W}_1$ also holds, by substituting \eqref{eq:p0tail} and \eqref{eq:muxtail} into \eqref{eq:W1gtrtail} and then using directly
\[\mathcal W_1(p_{nt_*},\mu) \leqslant \lambda(t_*)^n \mathcal W_1(p_0,\mu). \] 
To prove Theorem \ref{thm:noquantconv} for general $\mathcal{W}_p$, we use $\mathcal W_p(p_T,\mu) \geqslant \mathcal W_1(p_T,\mu)$ for all $p\geqslant 1$, so the same argument holds for $\mathcal{W}_p$ as long as we adjust $\beta$ such that $q_0\mu_x \in \mathcal{P}_p(\R^d)$. To see the dependency with respect to the initial condition, we can also use that
\[\mathcal W_p^p((1-\alpha)\mu+\alpha q\mu,\mu) \leqslant \alpha \mathcal W_p^p (q\mu,\mu)\,.  \] 

Finally, for $\mathcal V$-norms with $\mathcal V$ having at most polynomial growth, we see that this counter-example satisfies $\|p_0-\mu\|_{\mathcal V}<\infty$ when $\beta$ is large enough, and moreover
\[\|(1-\alpha)\mu + \alpha p_0 - \mu\|_{\mathcal V} = \alpha \|\mu - p_0\|_{\mathcal V}\,.\]
As for $\mathcal W_p$, this shows  that we can find an initial condition arbitrarily close to $\mu$ in terms of $\mathcal V$-norm with an algebraic lower-bound on $\|p_t-\mu\|_{\mathcal V}$, which proves the second inequality in~\eqref{eq:lowerboundsTheorem2}.
\end{example}

\begin{example}\label{ex:strexp}
    One can also prove Theorem \ref{thm:noquantconv} using stretched exponential initial conditions, which has arbitrary algebraic moments. Let $q_0\mu_x \sim \exp(-\bangle{x}^\delta)$ with $\delta\in (0,1)$, then we can estimate its tail
    \begin{equation}\label{eq:strexptail}
        \pp_{q_0\mu_x}(|X_0|>2Tv_*)\sim \int_{2Tv_*}^\infty r^{d-1}\exp(-r^\delta)\dd r \sim (2Tv_*)^{d-\delta} \exp(-(2Tv_*)^\delta),
    \end{equation} 
    which again decays subexponentially and much slower than $\pp_{\mu_x} (|x|> Tv_*) $ as $T\to\infty$. Hence the rest of the arguments in Example \ref{ex:fattail} still apply. This shows that an exponential moment in initial condition is necessary to get a long-time exponential convergence. In particular, reasoning as in the previous example, this shows that, for compactly supported velocities, the second inequality in~\eqref{eq:lowerboundsTheorem2} holds for any $\mathcal V$ with a stretched exponential bound.
\end{example}

\subsection{Proof of Theorem \ref{thm:noquantconv}: Gaussian Velocity}
While in this case, the velocity is no longer uniformly bounded for all times, $|V_t|$ never changes for straight line motions or bounces, and redraws from $\kappa$ with constant rate $\gamma$. Therefore, one may use Gaussian concentration inequality to show that seeing a velocity of large magnitude is extremely unlikely, so that our proof essentially reduces to the bounded velocity case. In fact, in view of Lemma \ref{lem:tail}, one only needs to control $\pp\big(\sup_{t\in[0,T]} |V_t|>v_*\big)$ for a suitable $v_*$ which depends on $T$.

\begin{lemma}\label{lem:Gaussiantail}
    Let $V_t$ be the velocity process for either ZZP or BPS, such that with rate $\gamma$, the velocity is redrawn from $\kappa$. For any $v_*>0$, define $p_* := \pp_{\kappa}(|v|>v_*)$, then for any $T>0$, we have
    \begin{equation}\label{eq:probbigV}
        \pp\po\sup_{t\in[0,T]} |V_t|>v_*\pf \leqslant (1+\gamma T) p_*.
    \end{equation}
\end{lemma}
\begin{proof}
Let $N_t$ be the number of velocity refreshments up to time $t$, then $\{N_t\}_{t\ge 0}$ is a Poisson process, or equivalently, for any fixed $T$, $N_T$ is a Poisson random variable with rate $\gamma T$. Equivalently,
\[ \pp(N_T=n) = \frac{(\gamma T)^n}{n!}\exp(-\gamma T).\]
Now, conditioning on $N_T=n$, there are almost surely $n+1$ different values of $|V_t|$ since $|V_t|$ does not change between velocity refreshments. Hence 
\[ \pp\po\sup_{t\in[0,T]} |V_t|>v_* \Big|N_T=n\pf = 1- \pp\po|V_t|\le v_*, \forall t\in[0,T] \Big|N_T=n\pf = 1-(1-p_*)^{n+1}.\]
We can therefore control the probability of high velocity by
\begin{align*}
     \pp\po\sup_{t\in[0,T]} |V_t|>v_*\pf & = \sum_{n=0}^\infty  \pp\po\sup_{t\in[0,T]} |V_t|>v_* \Big|N_T=n\pf \pp(N_T=n) \\ & = \sum_{n=0}^\infty \big(1-(1-p_*)^{n+1}\big)\frac{(\gamma T)^n}{n!}\exp(-\gamma T) \\ & = 1-(1-p_*)\exp(-\gamma T p_*) \le (1+\gamma T)p_*.
\end{align*}
\end{proof}
We now return to the proof of Theorem \ref{thm:noquantconv} when $\kappa$ is Gaussian. The idea is that, when we fix $T>0$, the choice of $v_*$ also increases with $T$, so that $(1+\gamma T)p_*$ decays much faster than $\pp_{q_0\mu_x}(|X_0|>2Tv_*)$ as $T\to\infty$. Note that 
\[\pp_\kappa(|v|>R) \sim \int_{|v|>R}\exp(-\frac{|v|^2}{2})\dd v \sim \int_{R}^\infty r^{d-1}\exp(-\frac{r^2}{2})\dd r \sim R^{d-2}\exp(-\frac{R^2}{2}), \quad \textrm{ as }R\to\infty, \]
thus if we pick $v_* = \log(1+T)$,  we obtain
\begin{equation}\label{eq:Vtbigalg} \pp\po\sup_{t\in[0,T]} |V_t|>v_*\pf < (1+\gamma T) v_*^{d-2}\exp(-\frac{v_*^2}{2}) \sim (1+\gamma T)\log^{d-2}(1+T)\exp\left(-\frac{\log^2(1+T)}{2}\right),\end{equation}
which decays to zero faster than any algebraic power of $T$ as $T\to \infty$. Meanwhile, for our $q_0\mu_x \propto \bangle{x}^{-(d+\beta)}$ given in Example \ref{ex:fattail}, we have 
\begin{equation}\label{eq:p0tailGauss} \pp_{q_0\mu_x}(|X_0|>2Tv_*)\sim  (2Tv_*)^{-\beta}\sim (T\log(1+T))^{-\beta}, \end{equation}
which still decays algebraically fast to zero. Substituting \eqref{eq:p0tailGauss}, \eqref{eq:Vtbigalg}, \eqref{eq:muxtail} into Lemma \ref{lem:tail}, and we finish the proof of Theorem \ref{thm:noquantconv} when $\kappa$ is Gaussian following the identical argument in Section \ref{subsec:bddv}. For stretched exponential $q_0\mu_x \propto \exp(-\bangle{x}^\delta)$ given in Example \ref{ex:strexp}, if we compare \eqref{eq:Vtbigalg} and \eqref{eq:strexptail}, we require $v_*^2 \gg (Tv_*)^\delta$, or $v_* \sim T^{\frac{\delta}{2-\delta}+\varepsilon}$ for our argument to hold, so exponential convergence still fails if $(Tv_*)^\delta \sim T^{\delta+\frac{\delta^2}{2-\delta}+\delta\varepsilon} \ll T$, which is the case when $\delta<\frac{2}{3}$. 

\subsection{Relation with previous results}\label{subsec:previous2}

 Let us now  review some previous results of exponential convergence for BPS and ZZP. Comparing these results to Theorem~\ref{thm:noquantconv} provides some insights on the question of long-time convergence of these samplers.

For the BPS, by \cite[Theorem 4]{durmus2020geometric}, if $\kappa$ admits a Gaussian moment, then there exists $a,b,C,\lambda>0$ such that
\begin{equation}
    \label{eq:V-ergodicity}
\forall \nu \in \mathcal P(\R^{2d}),\ t\geqslant 0,\quad \| \nu P_t - \mu\|_{\mathcal V} \leqslant C e^{-\lambda t } \|\nu - \mu\|_{\mathcal V}
\end{equation}
with Lyapunov function
\[\mathcal V (x,v) = e^{a |x|} + e^{b|v|^2} \,.\]
In particular, in dimension 1, this applies to the ZZP which is then the BPS with $\kappa=\frac12 (\delta_{-1}+\delta_1)$. In higher dimension, for the ZZP, we may use the results from~\cite{bierkens2019ergodicity}, which would give the result with $\mathcal V(x,v)= e^{a |x|^2}$ for some $a>0$ (cf. \cite[Lemma 11]{bierkens2019ergodicity}). However, with a standard Gaussian target, the coordinates of the ZZP are independent one-dimensiional ZZP, and thus the previous result from \cite{durmus2020geometric} or the specific studiy of ZZP in dimension 1 from \cite[Theorem 5]{10.1214/16-AAP1217} (itself based on the Lyapunov function of \cite{fontbona2016long}) gives the geometric $\mathcal V$-ergodicity~\eqref{eq:V-ergodicity} with
\[\mathcal V(x,v) = \sum_{i=1}^d e^{a |x_i|}\]
for some $a>0$. These results show that, in  Theorem~\ref{thm:noquantconv}, the restriction to $\mathcal V$-norms with at most stretched exponential with any $\delta<1$ is sharp.

\medskip

Concerning $L^2$ approaches for PDMP samplers, the DMS approach has been applied in \cite{andrieu2021hypocoercivity,MRZ}, and the space-time Poincar{\'e} inequality approach in \cite{lu2022explicit}. In the particular cases of BPS and ZZP with standard Gaussian target measures, all these works apply and yield the exponential convergence 
\[ \forall \nu \in \mathcal P(\R^{2d}),\ t\geqslant 0,\quad \| g_t - 1\|_{L^2(\mu)} \leqslant C e^{-\lambda t } \|g_0 - 1\|_{L^2(\mu)}  \]
for some constant $C,\lambda>0$. This implies exponential decay in entropy when $g_0 \in L^2(\mu)$. Besides, by Cauchy-Schwarz, $g_0\in L^2(\mu)$ implies that the initial distribution has some Gaussian moments, which is clearly violated by our Examples \ref{ex:fattail} and \ref{ex:strexp}.

\section{Open Questions}\label{sec:conclusion}
While our negative result in Theorem \ref{thm:noquantconv} sheds light on the scenarios where exponential convergence fails due to heavy tails of the initial condition $q_0\mu_x$, there are still various questions that would be of interest.
\begin{enumerate}
    \item Does exponential convergence 
    \[\KL(p_t\|\mu)\le C_0\exp(-\lambda t)\] hold for ZZP and BPS, if $p_0$ has exponential tails, and $\mu$ satisfies LSI, for some constant $C_0$ depending on $p_0$? This holds if $\chi^2(p_0\|\mu)<\infty$ since $L^2$ controls the entropy, and the exponential decay in $L^2$ is known \cite{andrieu2021hypocoercivity, lu2022explicit}. 
    \item Can one further weaken our non-convexity condition given in Assumption \ref{ass:JLass1}, namely, $\nabla^2 U\ge -K \mathrm{Id}$ with $K<\rho$?
    \item Our non-convergence result is stated with $t$ being the ``PDE time'', which is different from the ``event time'' or the number of bounces and refreshments that is more directly related to the sampling complexity of the PDMP algorithms. In particular, a PDMP for sampling starting far away from the center of $\mu_x$ will spend a long ``PDE time'' moving in a straight line towards the center without a bouncing event, but zero ``event time'' making the straight line move.  If we instead consider the event time, do we get exponential convergence, or if non-quantitative convergence still holds?
    \item If we allow velocity refreshment rates to depend on position, what convergence result can we obtain? Can we accelerate convergence by letting $\gamma=\gamma(x) \rightarrow \infty$ as $|x|\to\infty$? (Indeed, in this situation, the arguments from Section~\ref{sec:nonconvpdmp} may fail). 
    \item For the family of processes introduced in~\cite{MRZ}, which contains the forward event-chain sampler from~\cite{michel2020forward} and interpolates continuously from RHMC to BPS, the arguments from Section~\ref{sec:nonconvpdmp} do not apply since the norm of the velocity may change at ``bounces''  (whose rate goes to infinity with $|\nabla U(x)|$). Do these processses share the entropic speedup of RHMC, or the absence of exponential entropic convergence of BPS?
\end{enumerate}

\section*{Acknowledgement}
The research of PM is supported by the project CONVIVIALITY (ANR-23-CE40-0003) of
the French National Research Agency. Part of this work is completed when LW visited Universit\'e Gustave Eiffel. LW would like to thank Universit\'e Gustave Eiffel for their hospitality. We would like to thank Manon Michel for mentioning open question (3). Generative artificial intelligence is crucially involved in this work, as LW used ChatGPT 5.5 Pro to find a counterexample to an estimate we had hoped to obtain, which guides us towards the direction of negative result in Theorem \ref{thm:noquantconv}, as well as the regularity argument in the proof of Theorem \ref{thm:rhmc}.

\bibliographystyle{plain}
\bibliography{main}
\end{document}